\documentclass[a4paper,11pt]{article}
\usepackage{amssymb,bbm}
\usepackage{graphicx}
\usepackage{tikz}
\usetikzlibrary{intersections}
\usetikzlibrary{svg.path}
\usetikzlibrary{decorations.pathmorphing}

\usepackage{pgfplots}
\pgfplotsset{compat=1.11}
\usepackage{mathrsfs}
\usepackage{amsmath}
\usepackage{amsthm}
\usepackage{enumerate}
\usepackage{color}
\usepackage{verbatim}
\usepackage{todonotes}

\usepackage[margin=2.3cm, includefoot, footskip=30pt]{geometry}
\usepackage{array}
\newcolumntype{e}{>{\displaystyle}r @{\,} >{\displaystyle}c @{\,} >{\displaystyle}l}
\usepackage{amssymb,amsopn,amscd}

\usepackage[colorlinks]{hyperref}
\hypersetup{final}
\usepackage{xcolor}
\hypersetup{
    colorlinks,
    linkcolor={blue!40!black},
    citecolor={blue!40!black},
    urlcolor={blue!80!black}
}
\usepackage{nameref}

\usepackage[figure,table]{hypcap} % Correct a problem with hyperref

\theoremstyle{plain}
\newtheorem{Theorem}{Theorem}[section]
\newtheorem{Corollary}[Theorem]{Corollary}
\newtheorem{Lemma}[Theorem]{Lemma}

\newtheorem{Problem}[Theorem]{Problem}

\theoremstyle{definition}

\newtheorem{Remark}[Theorem]{Remark}

\newcommand{\E}{\mathbb{E}}

\newcommand{\R}{\mathbb{R}}
\newcommand{\N}{\mathbb{N}}

\renewcommand{\P}{\mathbb{P}}

\newcommand{\eps}{\varepsilon}
\newcommand{\MB}{\partial \mathcal{M}}
\newcommand{\MBh}{\partial \mathcal{M}_\mathrm{harm}}

\newcommand{\bE}{\mathbf{E}}
\newcommand{\bP}{\mathbf{P}}
\newcommand{\cM}{\mathcal{M}}
\newcommand{\bW}{\mathbf{W}}
\newcommand{\dif}{\mathrm{d}}

\numberwithin{equation}{section}

  \newcounter{constant}

\begin{document}

\title{On the boundary at infinity for branching random walk}

\author{Elisabetta Candellero\footnote{ecandellero@mat.uniroma3.it;  Roma Tre University, Dept.\ of Mathematics and Physics, Rome, Italy.} \and Tom Hutchcroft\footnote{t.hutchcroft@caltech.edu;  Division of Physics, Mathematics and Astronomy, California Institute of Technology}}

\maketitle

%%\begin{abstract}
%%Let $G$ denote the Cayley graph of a non-amenable finitely generated group.
%%We show that a supercritical (i.e., with positive chance of indefinite survival) and transient branching random walk on $G$ induces a \emph{random} measure on the Martin boundary for the underlying random walk.
%%In addition, we investigate the connection between such random measure and the limiting measure induced by the branching random walk, seen as a Markov chain on the space of configurations, on the corresponding Martin boundary.
%%We propose a conjecture that relates these two limiting measures and show an example where such a conjecture holds.
%%To our knowledge this is the first work investigating such a relationship between related limiting measures.
%%\end{abstract}

\begin{abstract}
We prove that supercritical branching random walk on a transient graph converges almost surely under rescaling to a random measure on the Martin boundary of the graph. Several open problems and conjectures about this limiting measure are presented.
% Let $G$ denote the Cayley graph of a non-amenable finitely generated group, and consider a supercritical (i.e., with positive chance of indefinite survival) and transient branching random walk (BRW) on $G$.
% Such BRW can be seen as a process on $G$, as well as a Markov chain on the set of its possible configurations.
% First we show that a ``renormalized'' version of BRW induces a \emph{random measure} on the Martin boundary for the underlying random walk.
% Subsequently, we analyze an example where such random measure generates the tail $\sigma$-algebra of the BRW seen as a Markov chain on the set of its configurations.
% We believe this to be a general fact, which we state in Conjecture \ref{conj:sigma}.
% %
% %
% %-----
% %
% %, we investigate the relations between these two points of view.
% %We show that in the first interpretation BRW induces a \emph{random} measure on the Martin boundary for the underlying random walk.
% %Meanwhile, classical Martin boundary theory guarantees that (in the second interpretation) BRW seen as a Markov chain induces a limiting measure on the Martin boundary of the configuration graph.
% %We investigate the connections between such limiting measures and conjecture that they induce the same tail sigma-algebra, we also present an example where the conjecture holds.
% To our knowledge, this is the first work comparing different Martin boundaries investigating how they are related to each other.
\end{abstract}

%\tableofcontents

%\section{Preliminaries}
\section{Introduction}

The asymptotic behavior of Markov chains on graphs is at the core of potential theory, which blends together tools from probability, algebra and analysis.
Historically, much of this theory has focused on \emph{reversible} Markov chains, with particular emphasis given to random walks on graphs and groups following the influential early work of Furstenberg \cite{Furstenberg63,Furstenberg71a,Furstenberg71b}. A cornerstone of this theory is the \emph{Poisson boundary}, a measure-theoretic object encoding all possible (positive probability) limiting behaviours of the random walk and which, by a foundational theorem of Blackwell \cite{blackwell1955transient}, may be defined equivalently either in terms of the invariant $\sigma$-algebra or the space of bounded harmonic functions on the Markov chain. The \emph{Martin boundary} is a \emph{topological} boundary that both recovers the measure-theoretic Poisson boundary and also includes  information about possible \emph{singular} (probability zero) behaviours of the walk at infinity.

While the Poisson and Martin boundaries of a Markov chain are rather abstract and difficult to study in general, they are known to coincide with concrete geometric compactifications in several classes of examples including hyperbolic groups \cite{ancona1987negatively,ancona,kaimanovich2000poisson}, planar graphs \cite{G13,ABGN14,10.1214/17-EJP116}, lamplighter groups \cite{kaimanovich_vershik,erschler2010poisson,lyons2015poisson}, and Diestel-Leader graphs \cite{woess2005lamplighters}. 
In general, however, the structure of the Martin boundary might be richer than what one would guess geometrically \cite{PicardelloWoessBitree}.
For random walks on groups, various important and beautiful works have established characterizations and criteria for the non-triviality of the Poisson and Martin boundaries in terms of speed, entropy, heat kernel decay and volume growth \cite{avez1974theoreme,kaimanovich_vershik,peres2020groups,amir2017every,frisch2019choquet}. See \cite{woess2000,LP:book,zheng2021asymptotic} for overviews of the literature.
% what is known now is related to the asymptotic behavior of \emph{random walks} on groups. 

% (See for example  \cite{woess2000} and \cite{Woess-denumerable} and references therein for a general construction, and also \cite{sawyer_harmonic_functions} for a more direct explanation.)
% Advances in this direction have appeared for example in the seminal paper \cite{kaimanovich_vershik}, where the authors manage to characterize intrinsic (algebraic) properties of a group by probabilistic means, relating the triviality of the \emph{Poisson boundary} of a random walk with \emph{amenability} of the group.
% (See \cite{kaimanovich_vershik} for the precise definitions and statements.)

In this work we aim to initiate a systematic investigation of the boundary theory of \emph{branching random walks}.
% More specifically, let $G=(V(G),E(G))$ denote the Cayley graph of a non-amenable finitely generated group.
Let $M=(S,P)$ be a Markov chain with countable state space $S$ and transition matrix $P$, and let $\mu$ be an \emph{offspring distribution}, that is, a probability measure supported on the non-negative integers.
Branching random walk (BRW) on $M$ is a Markov process governed by $P$ and $\mu$ in the following sense.
%Particles move independently according to the kernel $P$, and they reproduce (independently) according to the given offspring distribution $\nu$.
%
%The process is defined in a recursive way.
Choose any state $x\in S$ and at time $0$ place one alive particle at $x$.
Inductively, at each time step all alive particles reproduce independently of one another according to $\mu$ and then die, so that each particle is replaced by an independent random number of offspring with law $\mu$.
Subsequently, as part of the same time step, all (alive) particles take one step of the Markov chain, independently of one another, according to the transition kernel $P$. This defines a Markov process $(B_n)_{n\geq 0}$ with state space $\Omega = \{$finitely supported functions $S \to \N\}$, where $B_n(x)$ is the number of (alive) particles at the state $x$ at time $n$. One can also equivalently consider the branching random walk as a tree-indexed Markov chain where the underlying genealogical tree of the particles is itself a random Galton-Watson tree \cite{benjamini1994markov,MR1254826}. In particular, it follows from the classical theory of branching processes that the process dies almost surely if the mean $\bar \mu = \sum n \mu(n)$ satisfies $\bar \mu <1$ and survives forever with positive probability when $\bar \mu >1$. A detailed account of what is known about branching random walk on $\mathbb{Z}$, a rich subject that is largely disjoint from the kind of questions we consider here, can be found in \cite{shi2015branching}.
%We will investigate the asymptotic behavior of such a process in a sense that will be made clear below.

% Another way to think of BRW is the following.
% For all $n\geq 0$ and all realizations $\omega\in \Omega$ we denote by $\B_n(\omega)$ the configuration of BRW at time $n$. 
% More precisely, $\B_n(\omega)\in \N_{\geq 0}^{V(G)}$ is an infinite-dimensional vector indexed by the vertices of $G$, so that for each $v\in V(G)$ its $v$-th entry is the number of particles of BRW alive at $v$ at time $n$.
% Thus, one can think of BRW as a \emph{Markov chain} on the graph of configurations which we call $X$.
% We shall always consider transient BRW on $G$, hence the corresponding Markov chain on $X$ is transient too.

% The aim of this work is to initiate a new type of study on the limiting behavior of BRW by using these two different interpretations together with the tools available to potential theory.
% We show that a ``renormalized'' version of BRW induces a \emph{random measure} on the Martin boundary for the underlying random walk, and moreover we analyze an example where such random measure generates the tail $\sigma$-algebra of the BRW seen as a Markov chain on the set of its configurations.
% We believe this to be a general fact, which we state in Conjecture \ref{conj:sigma}.

% We now state our main theorem. 
Previous works on the behaviour at infinity of branching random walk have focused either on the geometric properties of the \emph{trace} (that is, the graph spanned by the edges crossed by the process) \cite{benjamini_mueller,candellero_gilch_mueller,candellero_roberts_BRW,Hutchcroft-NonIntersectionBRW} or the set of points accumulated to by the process in some geometric compactification. The latter approach has been carried out primarily in the setting of trees and hyperbolic spaces, where a detailed and sophisticated theory has now been developed \cite{MR1452555,MR1641015,hueter_lalley,lalley-sellke_hawkes,lalley2006weak,candellero_gilch_mueller,MR3087391,MR3194496,sidoravicius2020limit}.

Our main theorem states that branching random walks always converge under suitable rescaling to a random measure on the Martin boundary. We are optimistic that this random measure is both an interesting object of study and its own right and should lead to new perspectives and clarity  on the limit-set theory in the hyperbolic case. 
We write $\P_x$ and $\E_x$ for probabilities and expectations taken with respect to the law of the branching random walk started with a single particle at $x$ and write $\mathbf{P}_x$ and $\mathbf{E}_x$ for probabilities and expectations taken with respect to the underlying Markov chain started at $x$.

\begin{Theorem}
\label{thm:convergence}
Let $M=(S,P)$ be a transient, irreducible Markov chain, let $\mu$ be an offspring distribution with mean $\bar \mu>1$ satisfying the \emph{$L \log L$ condition} $\sum_{n=1}^\infty \mu(n)n\log n <\infty$, and let $B=(B_n)_{n\geq 0}$ be a branching random walk on $M$ started at some state $o$. Then $(\bar \mu)^{-n} B_n$ almost surely converges weakly  to a random measure $\mathbf{W}$ on the Martin compactification $\cM$ of $M$ that is supported on the harmonic Martin boundary of $M$ and satisfies
\begin{equation}
\label{eq:ExpectationIdentity}
\mathbb{E}_{o}\mathbf{W}(A)=\mathbf{P}_o(X_\infty \in A)
\end{equation}
for every Borel set $A \subseteq \mathcal{M}$.
\end{Theorem}

A selection of open problems regarding this random measure are discussed in Section~\ref{sec:problems}.

\medskip

The so-called `$L \log L$ condition' $\sum_{n=1}^\infty \mu(n)n\log n <\infty$ appearing here is needed for the limiting measure $\bW$ to be non-zero. Indeed, if $\mu$ is supercritical and $(|B_n|)_{n\geq 0}$ is the underlying  branching process of the branching random walk $(B_n)_{n\geq 0}$ then the limit
\[
W:= \lim_{n\to \infty} (\bar \mu)^{-n} |B_n|
\]
exists almost surely by the martingale convergence theorem. The Kesten-Stigum theorem \cite{KestenStigum66,LPP95} states that the following are equivalent:
\begin{enumerate}
  \item $W>0$ with positive probability.
  \item $W>0$ almost surely on the event that $(|B_n|)_{n\geq 0}$ survives forever.
  \item $\E W =1$.
  \item $(\bar \mu)^{-n} |B_n|$ is uniformly integrable.
  \item $\mu$ satisfies the $L\log L$ condition.
\end{enumerate}
Since any measure arising as the weak limit of $(\bar \mu)^{-n} B_n$ must have total mass $W$, it follows that if $\mu$ does \emph{not} satisfy the $L \log L$ condition then $(\bar \mu)^{-n} B_n$ almost surely converges weakly to the zero measure. (In particular, the conclusions of Theorem~\ref{thm:convergence} other than the formula \eqref{eq:ExpectationIdentity} hold vacuously in this case.)
\begin{Remark}
\label{remark:tail}
The Martin boundary of \emph{branching processes} (and hence of the Markov chain governing the \emph{number} of particles in a branching random walk) has been investigated in \cite{Overbeck} and \cite{Lootgieter}, whose results imply in particular that the random variable $W=\lim_{n\to\infty} (\bar\mu)^{-n}|B_n|$ generates the entire tail $\sigma$-algebra of the branching process $(|B_n|)_{n\geq 0}$. For more recent results we refer the reader to \cite{Abraham-Delmas} and references therein. While it is natural to guess from this that the random measure $\bW$ encodes all the tail information of the branching random walk $(B_n)_{n\geq 0}$, we observe in Section~\ref{sec:problems} that this is not the case even in very simple examples.
\end{Remark}
% In such works the authors analyze the ``behavior at infinity'' of some branching mechanisms in the three usual regimes (sub-critical, critical and super-critical), confirming that the most interesting regime is the super-critical one.

\begin{Remark}
While we were in the final stages of preparing this paper, Kaimanovich and Woess \cite{KW2022} posted a preprint to the arXiv establishing, in independent work, similar results on the convergence of branching random walks to random measures; their results allow one to consider compactifications other than the Martin compactification and branching Markov chains where the offspring distribution depends on the location of the particle.
\end{Remark}

\section{Background on the Martin boundary} 
We now briefly recall the definition and basic properties of the Martin boundary, referring the reader to \cite{woess2000} and \cite{dynkin1969boundary} for further details.
Let $P$ be the transition matrix of an irreducible, transient Markov chain $M=(S,P)$ on a countable state space $S$. 
% (When $M$ is irreducible the root state $o$ may be chosen arbitrarily.) 
Recall that a function $h: S \to \R$ that is either non-negative or bounded\footnote{Throughout this section, we use the assumption that our functions are either non-negative or bounded to guarantee that all sums appearing are well-defined in $(-\infty,\infty]$.} is said to be \textbf{harmonic} if $h(x) = \sum_{y\in S} P(x,y)h(y)$ for every $x\in S$ and is said to be \textbf{superharmonic} if $h(x) \geq \sum_{y\in S} P(x,y)h(y)$ for every $x\in S$. The space of all non-negative superharmonic functions on $S$ is denoted by $\mathcal{S}^+$.  For each $y \in S$, the \textbf{Martin kernel} $K(\,\cdot\,,y)$ is the non-negative superharmonic function defined by
\[
K(x,y) = \frac{G(x,y)}{G(o,y)},
\]
where $G(x,y) = \sum_{n\geq 0}P^n(x,y)$ is the Green's function. The \textbf{Martin compactification} $\mathcal{M}=\mathcal{M}(M)$  is defined to be the unique smallest compactification of the discrete set $S$ for which all the Martin kernels $K(x,\,\cdot\,)$ extend continuously. A detailed explanation of how this property uniquely determines $\mathcal{M}$ can be found in \cite[Section 24]{woess2000}, where it is also shown that $\mathcal{M}$ is a Polish space.
Thus, a sequence of states $(y_n)_{n\geq 0}$ converges in the Martin compactification if and only if the Martin kernels $K(x,y_n)$ converge for each $x\in S$. Note that this does indeed give a compact space since $G(o,y) \geq \sup_n P^n(o,x) G(x,y)$ for every $x,y\in S$ and hence $0\leq K(x,y) \leq (\sup_n P^n(o,x))^{-1}<\infty$ for every $x,y\in S$. The \textbf{Martin boundary} of $M$ is the space $\partial \mathcal{M} = \mathcal{M} \setminus S$. Note that if $\xi \in \partial{M}_M$ then we can define the Martin kernel $K(\,\cdot\,,\xi):= \lim_{n\to\infty} K(\,\cdot\,,y_n)$ where $(y_n)_{n\geq 0}$ is any sequence of states converging to $\xi$; the choice of sequence does not affect the limit thus obtained.
\medskip

The most important properties of the Martin boundary are as follows:
\begin{enumerate}
\item Convergence of the random walk \cite[Theorem 24.10]{woess2000}: If $(X_n)_{n\geq 0}$ is a trajectory of the Markov chain $M$ then $X$ converges almost surely to a limit point $X_\infty \in \partial \mathcal{M}$. Moreover, the Martin kernel $K(x,\,\cdot\,)$ is the Radon-Nikodym derivative of the law of $X_\infty$ under $\mathbf{P}_x$ with respect to the law of $X_\infty$ under $\mathbf{P}_o$ in the sense that
\[
\mathbf{E}_x \left[F(X_\infty)\right] = \mathbf{E}_o \left[F(X_\infty) K(x,X_\infty)\right]
\]
for every $x\in S$ and every continuous function $F:\partial \mathcal{M} \to [0,\infty)$.
\item Representation of positive harmonic functions \cite[Theorem 24.7]{woess2000}: If $h:S \to [0,\infty)$ is a non-negative harmonic function then there exists a finite Borel measure $\nu_h$ on $\partial \mathcal{M}$ such that
\[
h(x) = \int K(x,\xi) \mathrm{d} \nu_h(\xi)
\] 
for every $x\in S$. Conversely, if $\nu$ is a finite Borel measure on $\partial \mathcal{M}$ then $\int K(x,\xi) \mathrm{d} \nu(\xi)$ is a non-negative harmonic function on $S$. The function $h$ is bounded if and only if $\nu_h$ may be taken  to be absolutely continuous with respect to the law of $X_\infty$ with essentially bounded Radon-Nikodym derivative; in particular, if $h:S \to \R$ is a bounded harmonic function then there exists a bounded Borel function $\phi:\partial \mathcal{M} \to \R$ such that
\[
h(x) = \mathbf{E}_o \left[f(X_\infty) K(x,X_\infty)\right] = \mathbf{E}_x \left[f(X_\infty)\right]
\]
for every $x\in S$ \cite[Theorem  24.12]{woess2000}. In both cases, the representation is not unique in general but can be made unique by restricting the boundary data to be supported on the \emph{minimal} Martin boundary \cite[Theorem 24.9]{woess2000}.
\item Probabilistic Fatou Theorem \cite[Theorem 10]{dynkin1969boundary}: If $\phi:\MB(M)\to \R$ satisfies $\bE_o|\phi(X_\infty)|<\infty$ and $h(x)=\bE_x \phi(X_\infty)$ denotes the harmonic extension of $\phi$ to $S$ then
\[
h(X_n) \to \phi(X_\infty)
\] 
almost surely as $n\to\infty$.
\end{enumerate}

\begin{Remark}
Let us warn the reader that the Martin boundary can be quite ill-behaved in general, and in particular that the harmonic extension of a continuous function need not be continuous on the Martin compactification.
\end{Remark}

\section{Proof of the convergence theorem}

% \begin{Lemma}
% The algebra generated by the indicator functions $\{\mathbbm{1}_x\}_{x\in S}$ and the harmonic extensions of continuous functions on $\partial \mathcal{M}(M)$ is dense in $C(\mathcal{M}(M))$.
% \end{Lemma}

% \begin{proof}
% It suffices by the Stone-Weierstrass Theorem to prove that the 
% \end{proof}

% The kernel $K(x,y)$ continuous in $y$ for each $x\in S$.

% \[
% \langle h, B_n \rangle := \sum_{x \in S} h(x) B_n(x)
% \]

For the remainder of this section we fix a countable-state Markov chain $M=(S,P)$ and a vertex $o$ such that $\sup_n P^n(o,x)>0$ for every $x\in S$, fix an offspring measure $\mu$ satisfying the $L\log L$ condition, and let $B_n$ be a branching random walk on $M$ with offspring measure $\mu$, started with a single particle at $o$. We write $\bar B_n = (\bar \mu)^{-n} B_n$ for each $n\geq 0$ and, given a function $f:S\to \R$, write $\langle f,\bar B_n \rangle := \sum_{x\in S} f(x) \bar B_n(x)$.

\begin{Lemma}
\label{lem:martingale}
If $h:S\to \R$ is harmonic and is either bounded or non-negative then $( \langle h,\bar B_n\rangle)_{n\geq 0}$ is a martingale with respect to the filtration generated by $(B_n)_{n\geq 0}$. Similarly, 
if $h:S\to \R$ is superharmonic and is either bounded or non-negative then $( \langle h,\bar B_n\rangle)_{n\geq 0}$ is a supermartingale with respect to the filtration generated by $(B_n)_{n\geq 0}$.
\end{Lemma}

\begin{proof}[Proof of Lemma~\ref{lem:martingale}]
Let $\mathcal{F}_n$ be the $\sigma$-algebra generated by $(B_i)_{i=0}^n$. If $h:S\to \R$ is either bounded or non-negative then we have by linearity of expectation that
\[\E\left[ \langle h, \bar B_{n+1} \rangle \mid \mathcal{F}_n\right] = \sum_{x \in S} h(x) (\bar \mu)^{-n-1} \E\left[ B_{n+1}(x) \mid \mathcal{F}_n \right] = (\bar\mu)^{-n}\sum_{x,y\in S} B_n(y) P(y,x) h(x) = \langle Ph, \bar B_n \rangle \]
where in the second equality we have used that 
$\E\left[ B_{n+1}(x) \mid \mathcal{F}_n \right] = \bar \mu \sum_{y\in S} B_n(y) P(y,x)$ by linearity of expectation and the definition of branching random walk. The claim follows since $Ph=h$ when $h$ is harmonic and $Ph \leq h$ when $h$ is superharmonic.
\end{proof}

\begin{Corollary}
\label{cor:martingale_convergence}
If $h:S\to \R$ is superharmonic and is either bounded or non-negative then the sequence $\langle h,\bar B_n \rangle$ converges almost surely as $n\to\infty$.
\end{Corollary}

\begin{Lemma}
\label{lem:closed_upper}
Let $V \subseteq \MB$ be closed and let $h(x):=\bP_x(X_\infty \in V)$ be the harmonic extension of $\mathbbm{1}_V$ to $S$. Then
  \[
\sup\{\mathbf{W}(V) : \mathbf{W} \text{ is a subsequential weak limit of $\bar B_n$}\} \leq \lim_{n\to\infty} \langle h, \bar B_n\rangle 
\]
almost surely.
\end{Lemma}

\begin{proof}[Proof of Lemma~\ref{lem:closed_upper}]
% We have by the probabilistic Fatou theorem that $h(X_n) \to \mathbbm{1}(X_\infty)$
% \begin{enumerate} 
  % \item 
We have by continuity of measure that for each $\eps>0$ there exists an open neighbourhood $U$ of $V$ in $\mathcal{M}$ such that 
  \[\bP_o(X_n \text{ ever vists $U$}) \leq \bP_o(X_\infty \in V) + \eps.\]
(Here we are implicitly using that, since $\MB$ is metrisable, each closed set can be written as the intersection of countably many open sets, namely the $1/k$ neighbourhoods of the set in some compatible metric.)  Fix one such choice of $\eps$ and $U$ and let $f:S\to[0,1]$ be the function
  \[
f(x):= \bP_x(X_n \text{ ever visits $U$}).
  \]
  Note that $f$ is harmonic on $S\setminus U$, superharmonic on $S \cap U$, and satisfies $f\geq h$ and $f\geq \mathbbm{1}_U$ everywhere. Since $\langle f,\bar B_n\rangle$ is a non-negative supermartingale, it converges a.s., and since $U$ is open we have by the portmanteau theorem that
  \begin{align*}
\sup\{\mathbf{W}(V) : \mathbf{W} &\text{ is a subsequential weak limit of $\bar B_n$}\} \\&\hspace{3cm}\leq \sup\{\mathbf{W}(U) : \mathbf{W} \text{ is a subsequential weak limit of $\bar B_n$}\} \\&\hspace{3cm}\leq  \limsup_{n\to\infty}\langle \mathbbm{1}_U, \bar B_n\rangle  \leq \limsup_{n\to\infty}\langle f, \bar B_n\rangle = \lim_{n\to\infty} \langle f, \bar B_n\rangle.
  \end{align*}
  On the other hand, since $f-h$ is non-negative, superharmonic, and satisfies $f(o)-h(o)\leq \eps$ we also have that
  \[
\E_o \left[\lim_{n\to\infty} \langle f-h, \bar B_n\rangle\right] \leq f(o)-h(o) \leq \eps
  \]
  and hence by Markov's inequality that
  \begin{multline*}
\P\left(\sup\{\mathbf{W}(V) : \mathbf{W} \text{ is a subsequential weak limit of $\bar B_n$}\} \geq \lim_{n\to\infty} \langle h, \bar B_n\rangle + \delta\right)
\\ \leq 
\P\left(\lim_{n\to\infty} \langle f, \bar B_n\rangle \geq \lim_{n\to\infty} \langle h, \bar B_n\rangle + \delta\right)  \leq \frac{\eps}{\delta}
  \end{multline*}
  for every $\delta>0$. Since $\eps,\delta>0$ were arbitrary it follows that
  \[
\sup\{\mathbf{W}(V) : \mathbf{W} \text{ is a subsequential weak limit of $\bar B_n$}\} \leq \lim_{n\to\infty} \langle h, \bar B_n\rangle \text{ almost surely}
\]
%   \]
%   \[\E_o \left[\limsup_{n\to\infty}\langle f, \bar B_n\rangle\right]=\E_o \left[\lim_{n\to\infty}\langle f, \bar B_n\rangle\right] \leq h(o) \leq \eps\]
% and hence that
% \[
% \E_o \left[ \sup\{\mathbf{W}(U) : \mathbf{W} \text{ is a subsequential weak limit of $\bar B_n$}\} \right] \leq \E_o \left[\limsup_{n\to\infty}\langle h, \bar B_n\rangle\right] \leq \eps.
% \]
% Since $\eps>0$ was arbitrary, it follows that
% \[
% \E_o \left[ \sup\{\mathbf{W}(\cE) : \mathbf{W} \text{ is a subsequential weak limit of $\bar B_n$}\} \right] =0
% \]
as claimed.
% \end{enumerate}
\end{proof}

% \begin{Corollary}
% Let $U \subseteq \MB$ be open and let $h(x):=\bP_x(X_\infty \in U)$ be the harmonic extension of $\mathbbm{1}_U$ to $S$. The event ``every subsequential limit $\bW$ of $\bar B_n$ satisfies $\int \phi(\xi) \dif \bW(\xi) = \lim_{n\to\infty}\langle h, \bar B_n \rangle$'' has probability one.
% \end{Corollary}

We next apply this one-sided estimate for closed sets to deduce an \emph{equality} holding for each open set, closed set, and continuous function.

\begin{Lemma}
\label{lem:openclosedcontinuous}
\hspace{1cm}
\begin{enumerate}
    \item Let $A \subseteq \MB$ be either open or closed and let $h(x):=\bP_x(X_\infty \in U)$ be the harmonic extension of $\mathbbm{1}_U$ to $S$. The event ``every subsequential limit $\bW$ of $\bar B_n$ satisfies $\bW(A) = \lim_{n\to\infty}\langle h, \bar B_n \rangle$'' has probability one.
  \item
Let $\phi:\MB \to \R$ be continuous and let $h$ be the harmonic extension of $\phi$ to $S$. The event ``every subsequential limit $\bW$ of $\bar B_n$ satisfies $\int \phi(\xi) \dif \bW(\xi) = \lim_{n\to\infty}\langle h, \bar B_n \rangle$'' has probability one.
\end{enumerate}
\end{Lemma}

\begin{proof}[Proof of Lemma~\ref{lem:openclosedcontinuous}]
We begin with claim $1$. We first prove that if $A=U \subseteq \MB$ is open then 
  \begin{equation}
  \label{eq:open_upper}
\sup\{\mathbf{W}(U) : \mathbf{W} \text{ is a subsequential weak limit of $\bar B_n$}\} \leq \lim_{n\to\infty} \langle h, \bar B_n\rangle 
\end{equation}
almost surely,
where $h$ is the harmonic extension of $\mathbbm{1}_U$ to $S$. Since $\MB$ is metrisable, the open set $U$ may be written as a countable union of closed sets $U = \bigcup_{m\geq 1} V_m$. (For example, if $d$ is a compatible metric on $\MB$ then $U$ can be written as the union of the sequence of closed sets $V_m = \{\xi \in \MB: d(\xi,U^c) \geq 1/m\}$.) Fix one such representation $U=\bigcup_{m\geq 1} V_m$ and for each $m\geq 1$ let $h_m(x):= \bP_x(X_\infty \in V_m)$ be the harmonic extension of $\mathbbm{1}_{V_m}$ to $S$. Since $h_m\leq h$ for each $m\geq 1$, we have by continuity of measure and 
Lemma~\ref{lem:closed_upper} that
  \begin{align*}
&\sup\{\mathbf{W}(U) : \mathbf{W} \text{ a subsequential weak limit of $\bar B_n$}\}\\ 
&\hspace{6cm}=
\sup_m \sup\{\mathbf{W}(V_m) : \mathbf{W} \text{ a subsequential weak limit of $\bar B_n$}\} \\
&\hspace{6cm}\leq \sup_m \lim_{n\to\infty} \langle h_m, \bar B_n\rangle \leq \lim_{n\to\infty} \langle h, \bar B_n\rangle
\end{align*}
almost surely as claimed. Now, if $U$ is open then $U^c=\MB\setminus U$ is closed and the harmonic extension of $\mathbbm{1}_{U^c}$ to $S$ is given by $1-h$. Thus, it follows from Lemma~\ref{lem:closed_upper} that
\begin{equation}
\label{eq:1minusU}
\sup\{\mathbf{W}(U^c) : \mathbf{W} \text{ is a subsequential weak limit of $\bar B_n$}\} \leq \lim_{n\to\infty} \langle 1-h, \bar B_n\rangle
\end{equation}
almost surely, 
and since $\int 1 \dif \bW(\xi) = \lim_{n\to \infty} \langle 1,\bar B_n\rangle = W<\infty$ almost surely for every subsequential weak limit $\bW$ of $\bar B_n$, we can rearrange \eqref{eq:1minusU} to deduce that
\begin{equation}
  \label{eq:open_lower}
\inf\{\mathbf{W}(U) : \mathbf{W} \text{ is a subsequential weak limit of $\bar B_n$}\} \geq \lim_{n\to\infty} \langle h, \bar B_n\rangle 
\end{equation}
almost surely. Together \eqref{eq:open_upper} and \eqref{eq:open_lower} imply the open case of the claim, and the closed case follows by taking complements as in the proof of \eqref{eq:1minusU}.

\medskip

We now turn to the second claim.
We may assume without loss of generality that $\phi$ takes values in $[0,1]$. For each $t\in [0,1]$ let $V_t = \{\xi \in \MB: \phi(\xi) \geq t\}$ and let $h_t(x):=\bP_x(X_\infty \in V_t)$ be the harmonic extension of $\mathbbm{1}_{V_t}$ to $S$, so that $\phi = \int_0^1 \mathbbm{1}_{V_t} \dif t$ and $h=\int_0^1 h_t \dif t$. We have by Lemma~\ref{lem:closed_upper} that 
\begin{align}
&\sup\left\{\int \phi(\xi) \dif \mathbf{W}(\xi) : \mathbf{W} \text{ is a subsequential weak limit of $\bar B_n$}\right\}
\nonumber\\
&\hspace{5cm}\leq \int_0^1 \sup\left\{\bW(V_t) : \mathbf{W} \text{ is a subsequential weak limit of $\bar B_n$}\right\} \dif t 
\nonumber\\&\hspace{5cm}\leq \int_0^1 \lim_{n\to\infty} \langle h_t, \bar B_n \rangle \dif t \leq \lim_{n\to\infty} \langle h, \bar B_n \rangle,
\label{eq:phi1}
\end{align}
almost surely, 
where we have used Fatou's lemma in the final inequality. Since $\phi$ takes values in $[0,1]$ we can define a continuous function $(1-\phi):\MB\to [0,1]$, which has harmonic extension to $S$ given by $1-h$, and we have by symmetry that
\begin{equation}
\label{eq:1minusphi}
\sup\left\{\int (1-\phi(\xi)) \dif \mathbf{W}(\xi) : \mathbf{W} \text{ is a subsequential weak limit of $\bar B_n$}\right\} \leq \lim_{n\to\infty} \langle 1-h, \bar B_n \rangle.
\end{equation}
Since $\int 1 \dif \bW(\xi) = \lim_{n\to \infty} \langle 1,\bar B_n\rangle = W$ for every subsequential weak limit $\bW$ of $\bar B_n$, we can rearrange \eqref{eq:1minusphi} to deduce that
\begin{equation}
\inf\left\{\int \phi(\xi) \dif \mathbf{W}(\xi) : \mathbf{W} \text{ is a subsequential weak limit of $\bar B_n$}\right\} \geq \lim_{n\to\infty} \langle h,\bar B_n \rangle
\label{eq:phi2}
\end{equation}
almost surely. The claim follows from \eqref{eq:phi1} and \eqref{eq:phi2}.
\end{proof}

The open-set case of Lemma~\ref{lem:openclosedcontinuous} has the following immediate corollary.
Recall that the \textbf{harmonic Martin boundary} $\partial \mathcal{M}_\mathrm{harm}$ is defined to be the support of the law of $X_\infty$ under $\bP_o$, that is, the smallest closed subset of $\MB$ for which $\bP_o(X_\infty \in \MBh)=1$.

\begin{Corollary}
\label{cor:support}
The event ``every subsequential limit of $ \bar B_n$ is supported on $\partial \mathcal{M}_\mathrm{harm}$" has probability one.
\end{Corollary}

\begin{proof}
The complement $\MBh^c$ is open and has harmonic extension $h(x)=\bP_x(X_\infty \in \MBh^c) \equiv 0$ by definition. The claim therefore follows immediately from Lemma~\ref{lem:openclosedcontinuous}.
\end{proof}

We next use the separability of $\MBh$ to exchange the order of the quantifiers in the second item of Lemma~\ref{lem:openclosedcontinuous}.

\begin{Lemma}
\label{lem:quantifier_exchange}
The event\footnote{Although we call this an event, it is not clear that it is Borel measurable. Still, the claim makes sense as one about the set being co-null, and we avoid dwelling on the point further. In fact this set is co-analytic (i.e., the complement of a continuous image of a Borel set on a Polish space), and is therefore universally measurable (i.e., measurable with respect to the completion of any Borel measure) by Lusin's theorem \cite[Theorem 21.10]{kechris2012classical}. Similar remarks apply to several other `events' we consider throughout the rest of the proof.} ``for every continuous function $\phi: \MBh \to [0,1]$ with harmonc extension $h$ to $S$, the sequence $\langle h,\bar B_n \rangle$ converges and every subsequential limit $\bW$ of $\bar B_n$ satisfies
$\int \phi(\xi)\dif\bW(\xi) = \lim_{n\to\infty}\langle h,\bar B_n \rangle$" has probability one.
\end{Lemma}

\begin{proof}
The Martin compactification $\MB$ is compact and metrisable, and since $S$ is dense in $\MB$ it is separable also. Since $\MBh$ is a closed subset of $\MB$, it is also a compact, separable, metrisable space. As such, the space $C(\MBh)$ of continuous functions on $\MBh$ is separable also. (Indeed, if $d$ is a metric compatible with the topology of $\MBh$ and $Q$ is a countable dense subset of $\MBh$ then the $\mathbb{Q}$-algebra generated by the continuous functions of the form $\{d(\,\cdot\,,q) : q \in Q\}$ is a countable dense subset of $C(\MBh)$.) 

\medskip

Let $A$ be a countable dense subset of $C(\MBh)$. We have by Corollary~\ref{cor:martingale_convergence} and Lemma~\ref{lem:openclosedcontinuous} that the event 
$\Omega=\{W=\lim_{n\to\infty}\langle 1,\bar B_n \rangle$ exists and is finite, and for each function $\phi\in A$ with harmonic extension $h$ the sequence $(\langle h,\bar B_n \rangle)_{n\geq 0}$ converges and every subsequential weak limit $\bW$ of $\bar B_n$ satisfies $\int \phi(\xi)\dif\bW(\xi)=\lim_{n\to\infty}\langle h,\bar B_n \rangle\}$ has probability one. 

\medskip

Now, if $\psi$ is an \emph{arbitrary} continuous function on $\MB$, then for each $\eps>0$ there exists $\phi\in A$ with $|\phi-\psi|\leq \eps$, so that if $h$ and $h'$ denote the harmonic extensions of $h$ and $h'$ to $S$ then $|h-h'|\leq \eps$. Thus, on the event $\Omega$ we have that
\begin{align*}
&\sup\left\{\int \psi(\xi) \dif \bW(\xi) : \bW \text{ a subsequential limit of $\bar B_n$}\right\} 
\\&\hspace{5cm}\leq \eps W + \sup\left\{\int \phi(\xi) \dif \bW(\xi) : \bW \text{ a subsequential limit of $\bar B_n$}\right\}
\\&\hspace{5cm}= \eps W + \lim_{n\to\infty} \langle h',\bar B_n \rangle \leq 2\eps W + \liminf_{n\to\infty} \langle h,\bar B_n \rangle,
\end{align*}
where the final inequality follows since $|\langle h-h',\bar B_n \rangle| \leq \eps \langle 1,\bar B_n \rangle$ for every $n\geq 0$.
We also have symmetrically that
\begin{align*}
&\inf\left\{\int \psi(\xi) \dif \bW(\xi) : \bW \text{ a subsequential limit of $\bar B_n$}\right\} 
\\&\hspace{5cm}\geq -\eps W + \inf\left\{\int \phi(\xi) \dif \bW(\xi) : \bW \text{ a subsequential limit of $\bar B_n$}\right\}
\\&\hspace{5cm}= -\eps W + \lim_{n\to\infty} \langle h',\bar B_n \rangle \geq -2\eps W + \limsup_{n\to\infty} \langle h,\bar B_n \rangle,
\end{align*}
yielding that the chain of inequalities
\begin{align*}
-2\eps W + \limsup_{n\to\infty} \langle h,\bar B_n \rangle &\leq \inf\left\{\int \psi(\xi) \dif \bW(\xi) : \bW \text{ a subsequential limit of $\bar B_n$}\right\} \\&\leq\sup\left\{\int \psi(\xi) \dif \bW(\xi) : \bW \text{ a subsequential limit of $\bar B_n$}\right\} \\
&\leq 2\eps W + \liminf_{n\to\infty} \langle h,\bar B_n \rangle.
\end{align*}
holds pointwise on the event $\Omega$. The claim follows since $\psi$ and $\eps>0$ were arbitrary.
\end{proof}

We may now conclude the proof of the main theorem. Note that we have not yet used the $L\log L$ condition. This will be needed only to verify the identity \eqref{eq:ExpectationIdentity}; if the $L\log L$ condition does \emph{not} hold then $\bar B_n$ converges to the zero measure!

\begin{proof}[Proof of Theorem~\ref{thm:convergence}]
We have by Corollary~\ref{cor:support} and Lemma~\ref{lem:quantifier_exchange} that the event ``
every subsequential weak limit of $\bar B_n$ is supported on $\MBh$, and if $\phi:\MBh\to \R$ is continuous then $\int \phi(\xi)\dif\bW_1(\xi)=\int \phi(\xi)\dif \bW_2(\xi)$ for any two subsequential weak limits $\bW_1$, $\bW_2$ of $\bar B_n$'' holds almost surely. (Again, it is very important that the quantifier over all continuous functions and subsequential limits is \emph{inside} the event being defined!)
 % that, almost surely, every subsequential weak limit of $\bar B_n$ is supported on $\MBh$, while Lemma~\ref{lem:quantifier_exchange} implies that, almost surely, every subsequential weak limit of $\bar B_n$ has the same integral against every continuous function on $\MBh$. 
% (Be careful to note that the ``almost surely" precedes the quantifiers in both statements! {\color{red}{I think it is important to note it, so I would just rephrase the sentence to make it a bit more formal...}})
 The Riesz-Markov-Kakutani representation theorem implies that measures on a compact metric space are determined by their integrals against continuous functions, so that all subsequential weak limits of $\bar B_n$ are identical almost surely. Since the space of measures on $\mathcal{M}$ is a compact metrisable space under the weak topology, it follows that $\bar B_n$ converges almost surely to a (random) limit measure $\bW$ as claimed. 

Now, if the offspring measure $\mu$ satisfies the $L\log L$ condition then we have by the Kesten-Stigum theorem that the martingale $(\langle 1,\bar B_n \rangle)_{n\geq 0}$ is uniformly integrable. As such, if the $L\log L$ condition holds then $(\langle h,B_n \rangle)_{n\geq 0}$ is a uniformly integrable martingale for each bounded harmonic function $h$ and we have by Lemma~\ref{lem:openclosedcontinuous} that if $\phi: \MB\to \R$ is a continuous with harmonic extension $h$ then
\[
\E \left[\int \phi(\xi)\dif\bW(\xi)\right] = \E \left[\lim_{n\to\infty} \langle h,\bar B_n \rangle\right] = \langle h,\bar B_0\rangle = \bE_o \left[ \phi(X_\infty)\right].
\]
This implies that the expectation of $\bW$ and the law of $X_\infty$ determine the same measure on $\MB$, which is equivalent to the claimed equality \eqref{eq:ExpectationIdentity}.
\end{proof}

\section{Open problems}
\label{sec:problems}

As mentioned in Remark~\ref{remark:tail}, since $W$ generates the tail $\sigma$-algebra of the branching process $(|B_n|)_{n\geq 0}$ \cite{Lootgieter,Overbeck} and $X_\infty$ generates the invariant $\sigma$-algebra of the underlying random walk, it is natural to wonder whether the random measure $\bW$ generates the tail $\sigma$-algebra of the branching random walk. Unfortunately this is not the case even in some very simple examples. Indeed, consider branching random walk on $\mathbb{Z}^3$ with deterministic branching governed by the degenerate offspring measure $\mu(16)=1$. Since the Martin boundary of $\mathbb{Z}^3$ consists of a single point $\{\mathbf{1}\}$ and $W=16$, $\bW$ is deterministically equal to a point mass of mass $16$ at $\mathbf{1}$. On the other hand, if we start with a single particle at the origin, then the number of particles at $(n,0,0)$ at time $n$ is itself a branching process with Binomial$(16,1/8)$ offspring distribution. Since this offspring distribution is supercritical but has positive probability to die after a single step, it survives forever with probability strictly between $0$ and $1$. As such, $\{B_n(n,0,0) >0$ for infinitely many $n\}$ is a non-trivial tail event that is not in the $\sigma$-algebra generated by $\bW$. Similarly, $\{
\max\{x:B_n(x,0,0)>0\} > - \max \{x: B_n(-x,0,0)>0\}$ for all sufficiently large $n\}$ is a non-trivial \emph{invariant event} for the branching random walk that does not belong to the $\sigma$-algebra generated by $\bW$. This examples suggests that the following problem is non-trivial.

\begin{Problem}
Characterise the $\sigma$-algebra generated by the limiting measure $\bW$.
\end{Problem}

As discussed in the introduction, much of the existing literature on the boundary behaviour of branching random walk focuses on the set of accumulation points of the walk \cite{MR1452555,MR1641015,sidoravicius2020limit}, and it would be interesting to reinterpret these analyses through the lens of Theorem~\ref{thm:convergence}. To do this, an important first step would be to compare the set of limit points with the support of $\bW$. Again, simple examples show that these are \emph{not} always the same: If one performs a supercritical branching random walk on the lamplighter group $\mathbb{Z}_2 \wr \mathbb{Z}$ then every vertex is visited infinitely often, so that the set of limit points is the entire Martin boundary, which has more than one point by, say, the results of \cite{amir2017every}. Meanwhile, the measure $\bW$ is supported on the harmonic Martin boundary, which in this case is the single point $\{\mathbf{1}\}$ since $\mathbb{Z}_2 \wr \mathbb{Z}$ is Liouville. Still, we expect that in `sufficiently nice' examples, such as hyperbolic groups, the two sets will coincide.

\begin{Problem}
Give conditions under which the support of $\bW$ coincides almost surely with the set of accumulation points of $B$.
\end{Problem}

% We expect that the support of $\bW$ always coincides with the set of limit points for branching random walk on a hyperbolic group.

For hyperbolic groups there are many further questions one could ask about the limiting measure $\bW$, mirroring what is known about the limit set. For example one can try to compute the Hausdorff dimension of the measure, or study the tail probability that an unusually large mass is placed on a small ball.

\begin{Problem}
Study the fine properties of the limiting measure $\bW$ for branching random walk on a hyperbolic group. 
% Does the measure have any particularly interesting behaviour at the recurrence/transience threshold?
\end{Problem}

Much of the interest in branching random walk on \emph{nonamenable groups} stems from the fact that there are \emph{two} phase transitions: The transition between certain death and possible survival when $\bar \mu = 1$ and the transition between transience and recurrence when $\bar \mu = 1/\rho$, where $\rho$ is the spectral radius of the random walk: when $\bar \mu \leq 1/\rho$ each point is visited at most finitely often almost surely while when $\bar \mu > 1/\rho$ every point is visited infinitely often almost surely on the event that the walk survives forever \cite{MR1254826,bertacchi_zucca,gantert_mueller}. It is likely that the answers to each of the questions above are different in the two cases $\bar \mu \leq 1/\rho$ and $\bar \mu > 1/\rho$, with the possibility of particularly interesting behaviour at or near the boundary case $\bar \mu = 1/\rho$.

It is also natural to wonder how this phase transition manifests itself in the behaviour of the limiting measure $\bW$. One plausible scenario is as follows: For random walks on nonamenable groups, the random measure $\bW$ is always singular with respect to harmonic measure (i.e.\ the law of $X_\infty$), but is, in some qualitative sense, `more singular' in the transient regime $\bar \mu \leq 1/\rho$ than in the recurrent regime $\bar \mu > 1/\rho$. While we are not confident enough to make a formal conjecture, it is also plausible that $\bW$ is almost surely supported on a strict subset of $\MBh$ whenever $\bar \mu \leq 1/\rho$. This is true for hyperbolic groups by the results of \cite{sidoravicius2020limit}. The results of \cite{Hutchcroft-NonIntersectionBRW} also suggest that the support of $\bW$ is always totally disconnected in the transient regime.

\subsection*{Acknowledgments} 
This work was initiated while TH was a senior research associate at the University of Cambridge, during which time he was supported by ERC starting grant  804166 (SPRS).
EC was supported by the project ``Programma per Giovani Ricercatori Rita Levi Montalcini'' awarded by the Italian Ministry of Education.
EC also acknowledges partial support by ``INdAM--GNAMPA Project 2019'' and ``INdAM--GNAMPA Project 2020''.

\footnotesize{

\nocite{}
\bibliographystyle{amsalpha}
\bibliography{bibliography2}

\providecommand{\bysame}{\leavevmode\hbox to3em{\hrulefill}\thinspace}
\providecommand{\MR}{\relax\ifhmode\unskip\space\fi MR }
% \MRhref is called by the amsart/book/proc definition of \MR.
\providecommand{\MRhref}[2]{%
  \href{http://www.ams.org/mathscinet-getitem?mr=#1}{#2}
}
\providecommand{\href}[2]{#2}
\begin{thebibliography}{ABGGN16}

\bibitem[ABGGN16]{ABGN14}
Omer Angel, Martin~T. Barlow, Ori Gurel-Gurevich, and Asaf Nachmias,
  \emph{Boundaries of planar graphs, via circle packings}, Ann. Probab.
  \textbf{44} (2016), no.~3, 1956--1984. \MR{3502598}

\bibitem[AD19]{Abraham-Delmas}
Romain Abraham and Jean-Fran\c{c}ois Delmas, \emph{Asymptotic properties of
  expansive {G}alton-{W}atson trees}, Electron. J. Probab. \textbf{24} (2019),
  Paper No. 15, 51. \MR{3916335}

\bibitem[AK17]{amir2017every}
Gideon Amir and Gady Kozma, \emph{Every exponential group supports a positive
  harmonic function}, arXiv preprint arXiv:1711.00050 (2017).

\bibitem[Anc87]{ancona1987negatively}
Alano Ancona, \emph{Negatively curved manifolds, elliptic operators, and the
  martin boundary}, Annals of mathematics \textbf{125} (1987), no.~3, 495--536.

\bibitem[Anc88]{ancona}
\bysame, \emph{Positive harmonic functions and hyperbolicity}, Potential
  theory---surveys and problems ({P}rague, 1987), Lecture Notes in Math., vol.
  1344, Springer, Berlin, 1988, pp.~1--23. \MR{973878}

\bibitem[Ave74]{avez1974theoreme}
Andr{\'e} Avez, \emph{Theoreme de choquet-deny pour les groupes a croissance
  non exponentielle.}

\bibitem[Bla55]{blackwell1955transient}
David Blackwell, \emph{On transient markov processes with a countable number of
  states and stationary transition probabilities}, The Annals of Mathematical
  Statistics (1955), 654--658.

\bibitem[BM12]{benjamini_mueller}
Itai Benjamini and Sebastian M{\"u}ller, \emph{On the trace of branching random
  walks}, Groups Geom. Dyn. \textbf{6} (2012), no.~2, 231--247. \MR{2914859}

\bibitem[BP94a]{benjamini1994markov}
Itai Benjamini and Yuval Peres, \emph{Markov chains indexed by trees}, The
  annals of probability (1994), 219--243.

\bibitem[BP94b]{MR1254826}
\bysame, \emph{Tree-indexed random walks on groups and first passage
  percolation}, Probab. Theory Related Fields \textbf{98} (1994), no.~1,
  91--112. \MR{1254826}

\bibitem[BZ08]{bertacchi_zucca}
Daniela Bertacchi and Fabio Zucca, \emph{Critical behaviors and critical values
  of branching random walks on multigraphs}, J. Appl. Probab. \textbf{45}
  (2008), no.~2, 481--497. \MR{2426846 (2009k:60206)}

\bibitem[CGM12]{candellero_gilch_mueller}
Elisabetta Candellero, Lorenz~A. Gilch, and Sebastian M{\"u}ller,
  \emph{Branching random walks on free products of groups}, Proc. Lond. Math.
  Soc. (3) \textbf{104} (2012), no.~6, 1085--1120. \MR{2946082}

\bibitem[CR15]{candellero_roberts_BRW}
Elisabetta Candellero and Matthew~I. Roberts, \emph{The number of ends of
  critical branching random walks}, ALEA Lat. Am. J. Probab. Math. Stat.
  \textbf{12} (2015), no.~1, 55--67. \MR{3333735}

\bibitem[Dyn69]{dynkin1969boundary}
Evgenii~Borisovich Dynkin, \emph{Boundary theory of markov processes (the
  discrete case)}, Russian Mathematical Surveys \textbf{24} (1969), no.~2, 1.

\bibitem[Ers10]{erschler2010poisson}
Anna Erschler, \emph{Poisson--furstenberg boundary of random walks on wreath
  products and free metabelian groups}, Commentarii Mathematici Helvetici
  \textbf{86} (2010), no.~1, 113--143.

\bibitem[FHTF19]{frisch2019choquet}
Joshua Frisch, Yair Hartman, Omer Tamuz, and Pooya~Vahidi Ferdowsi,
  \emph{Choquet-deny groups and the infinite conjugacy class property}, Annals
  of Mathematics \textbf{190} (2019), no.~1, 307--320.

\bibitem[Fur63]{Furstenberg63}
Harry Furstenberg, \emph{A {P}oisson formula for semi-simple {L}ie groups},
  Ann. of Math. (2) \textbf{77} (1963), 335--386. \MR{0146298}

\bibitem[Fur71a]{Furstenberg71a}
\bysame, \emph{Boundaries of {L}ie groups and discrete subgroups}, Actes du
  {C}ongr\`es {I}nternational des {M}ath\'ematiciens ({N}ice, 1970), {T}ome 2,
  Gauthier-Villars, Paris, 1971, pp.~301--306. \MR{0430160}

\bibitem[Fur71b]{Furstenberg71b}
\bysame, \emph{Random walks and discrete subgroups of {L}ie groups}, Advances
  in {P}robability and {R}elated {T}opics, {V}ol. 1, Dekker, New York, 1971,
  pp.~1--63. \MR{0284569}

\bibitem[Geo16]{G13}
Agelos Georgakopoulos, \emph{The boundary of a square tiling of a graph
  coincides with the {P}oisson boundary}, Invent. Math. \textbf{203} (2016),
  no.~3, 773--821. \MR{3461366}

\bibitem[GL13]{MR3087391}
S\'ebastien Gou\"ezel and Steven~P. Lalley, \emph{Random walks on co-compact
  {F}uchsian groups}, Ann. Sci. \'Ec. Norm. Sup\'er. (4) \textbf{46} (2013),
  no.~1, 129--173 (2013). \MR{3087391}

\bibitem[GM06]{gantert_mueller}
Nina Gantert and Sebastian M{\"u}ller, \emph{The critical branching {M}arkov
  chain is transient}, Markov Process. Related Fields \textbf{12} (2006),
  no.~4, 805--814. \MR{2284404 (2008c:60082)}

\bibitem[Gou14]{MR3194496}
S\'ebastien Gou\"ezel, \emph{Local limit theorem for symmetric random walks in
  {G}romov-hyperbolic groups}, J. Amer. Math. Soc. \textbf{27} (2014), no.~3,
  893--928. \MR{3194496}

\bibitem[HL00]{hueter_lalley}
Irene Hueter and Steven~P. Lalley, \emph{Anisotropic branching random walks on
  homogeneous trees}, Probab. Theory Related Fields \textbf{116} (2000), no.~1,
  57--88.

\bibitem[HP17]{10.1214/17-EJP116}
Tom Hutchcroft and Yuval Peres, \emph{{Boundaries of planar graphs: a unified
  approach}}, Electronic Journal of Probability \textbf{22} (2017), no.~none, 1
  -- 20.

\bibitem[Hut20]{Hutchcroft-NonIntersectionBRW}
Tom Hutchcroft, \emph{Non-intersection of transient branching random walks},
  Probab. Theory Related Fields \textbf{178} (2020), no.~1-2, 1--23.
  \MR{4146533}

\bibitem[Kai00]{kaimanovich2000poisson}
Vadim~A Kaimanovich, \emph{The poisson formula for groups with hyperbolic
  properties}, Annals of Mathematics (2000), 659--692.

\bibitem[Kec12]{kechris2012classical}
Alexander Kechris, \emph{Classical descriptive set theory}, vol. 156, Springer
  Science \& Business Media, 2012.

\bibitem[KPS98]{MR1641015}
F.~I. Karpelevich, E.~A. Pechersky, and Yu.~M. Suhov, \emph{A phase transition
  for hyperbolic branching processes}, Comm. Math. Phys. \textbf{195} (1998),
  no.~3, 627--642. \MR{1641015}

\bibitem[KS66]{KestenStigum66}
H.~Kesten and B.~P. Stigum, \emph{A limit theorem for multidimensional
  {G}alton-{W}atson processes}, Ann. Math. Statist. \textbf{37} (1966),
  1211--1223. \MR{0198552}

\bibitem[KV83]{kaimanovich_vershik}
Vadim~A. Ka{\u\i}manovich and Anatoly~M. Vershik, \emph{Random walks on
  discrete groups: boundary and entropy}, Ann. Probab. \textbf{11} (1983),
  no.~3, 457--490. \MR{704539 (85d:60024)}

\bibitem[KW22]{KW2022}
Vadim~A. Kaimanovich and Wolfgang Woess, \emph{Limit distributions of branching
  markov chains}, arXiv preprint arXiv:2205.13615 (2022).

\bibitem[Lal06]{lalley2006weak}
Steven~P Lalley, \emph{The weak/strong survival transition on trees and
  nonamenable graphs}, International Congress of Mathematicians, vol.~3, 2006,
  pp.~637--647.

\bibitem[Loo77]{Lootgieter}
J.~C. Lootgieter, \emph{La {$\sigma $}-alg{\`e}bre asymptotique d'une
  cha{\^{\i}}ne de {G}alton-{W}atson}, Ann. Inst. H. Poincar{\'e} Sect. B
  (N.S.) \textbf{13} (1977), no.~3, 193--230. \MR{0471099}

\bibitem[LP15]{lyons2015poisson}
Russell Lyons and Yuval Peres, \emph{Poisson boundaries of lamplighter groups:
  proof of the kaimanovich-vershik conjecture}, arXiv preprint arXiv:1508.01845
  (2015).

\bibitem[LP16]{LP:book}
\bysame, \emph{Probability on trees and networks}, Cambridge Series in
  Statistical and Probabilistic Mathematics, vol.~42, Cambridge University
  Press, New York, 2016, Available at \url{http://pages.iu.edu/~rdlyons/}.
  \MR{3616205}

\bibitem[LPP95]{LPP95}
Russell Lyons, Robin Pemantle, and Yuval Peres, \emph{Conceptual proofs of
  {$L\log L$} criteria for mean behavior of branching processes}, Ann. Probab.
  \textbf{23} (1995), no.~3, 1125--1138. \MR{1349164}

\bibitem[LS97]{MR1452555}
Steven~P. Lalley and Tom Sellke, \emph{Hyperbolic branching {B}rownian motion},
  Probab. Theory Related Fields \textbf{108} (1997), no.~2, 171--192.
  \MR{1452555}

\bibitem[LS00]{lalley-sellke_hawkes}
Steven~P. Lalley and Thomas Sellke, \emph{An extension of {H}awkes' theorem on
  the {H}ausdorff dimension of a {G}alton-{W}atson tree}, Probab. Theory
  Related Fields \textbf{116} (2000), no.~1, 41--56. \MR{1736589}

\bibitem[Ove94]{Overbeck}
Ludger Overbeck, \emph{Martin boundaries of some branching processes}, Ann.
  Inst. H. Poincar\'{e} Probab. Statist. \textbf{30} (1994), no.~2, 181--195.
  \MR{1276996}

\bibitem[PW94]{PicardelloWoessBitree}
Massimo~A. Picardello and Wolfgang Woess, \emph{The full {M}artin boundary of
  the bi-tree}, Ann. Probab. \textbf{22} (1994), no.~4, 2203--2222.
  \MR{1331221}

\bibitem[PZ20]{peres2020groups}
Yuval Peres and Tianyi Zheng, \emph{On groups, slow heat kernel decay yields
  liouville property and sharp entropy bounds}, International Mathematics
  Research Notices \textbf{2020} (2020), no.~3, 722--750.

\bibitem[Shi15]{shi2015branching}
Zhan Shi, \emph{Branching random walks}, Springer, 2015.

\bibitem[SWX20]{sidoravicius2020limit}
Vladas Sidoravicius, Longmin Wang, and Kainan Xiang, \emph{Limit set of
  branching random walks on hyperbolic groups}, arXiv preprint arXiv:2007.13267
  (2020).

\bibitem[Woe00]{woess2000}
Wolfgang Woess, \emph{Random walks on infinite graphs and groups}, Cambridge
  Tracts in Math., vol. 138, Cambridge University Press, Cambridge, 2000.

\bibitem[Woe05]{woess2005lamplighters}
\bysame, \emph{Lamplighters, diestel--leader graphs, random walks, and harmonic
  functions}, Combinatorics, Probability and Computing \textbf{14} (2005),
  no.~3, 415--433.

\bibitem[Zhe21]{zheng2021asymptotic}
Tianyi Zheng, \emph{Asymptotic behaviors of random walks on countable groups}.

\end{thebibliography}
%\printbibliography[heading=bibintoc]

}

\end{document}